\documentclass[a4paper,10pt]{article}
\usepackage{graphicx} 

\usepackage{orcidlink}
\usepackage{authblk}

\usepackage{amsfonts}
\usepackage{amsmath}
\usepackage{amssymb}
\usepackage{amsthm}

\newtheorem{lemma}{Lemma}
\newtheorem{theorem}{Theorem}
\newtheorem{corollary}{Corollary}
\newtheorem{definition}{Definition}

\newtheorem*{remark}{Remark}
\newtheorem{proposition}{Proposition}

\newcommand{\rrbracket}{\right]\hspace{-0.15em}\right]}
\newcommand{\llbracket}{\left[\hspace{-0.15em}\left[}
\newcommand{\interpret}[1]{\llbracket #1 \rrbracket}

\title{Inexpressibility in Exp-Minus-Log}
\author[1]{Mark Carney \orcidlink{0000-0001-9372-9033} \\ \small mark@quantumvillage.org}
\affil[1]{Quantum Village Inc.}
\date{April 2026}

\begin{document}

\maketitle

\begin{abstract}
    Odrzywo\l{}ek defined a system Exp-Minus-Log (EML) \cite{Odrzywolek2026} that reduces all elementary functions over complex numbers down to a constant `$1$', and a single two place function $E(\alpha, \beta) = \exp(\alpha) - \log(\beta)$. This paper shows that in this system, equivalent to Chow's EL numbers \cite{Chow1999}, every EML-expressible number is computable. We go on to prove that the canonical example of a non-computable real, Chaitin's $\Omega_U$, is inexpressible in EML. This gives a formal inexpressibility theorem for this system.
\end{abstract}

\section{Introduction}

The original paper due to Odrzywo\l{}ek \cite{Odrzywolek2026} presents a way to express all elementary functions, constants, and operations through a single binary operator $E(\alpha, \beta) = \exp(\alpha) - \log(\beta)$ with the constant `$1$'. This is at once surprising and aesthetically pleasing. Indeed, such closed-form number systems are useful in symbolic computation \cite{Chow1999}, CAS systems \cite{Richardson1969}, exact arithmetic \cite{Weihrauch2000}, etc.

Earlier work due to Chow \cite{Chow1999} created the EL numbers that are expressible through the use of $\exp$ and $\log$. Chow introduced the EL numbers to formalize the intuitive notion of ``numbers expressible in closed form'' \cite{Chow1999}. We show an equivalence with EML numbers.

The content of \cite{Odrzywolek2026} is mostly constructive and existential, and leaves open the question about any formal inexpressibility theorem. This paper fills that gap by showing the computability of the EML numbers, and thereby using Chaitin's $\Omega_U$ as the ``canonical'' non-computable witness - left-computably enumerable but not computable - thereby being a witness to inexpressibility in EML. Moreover, we prove this inclusion: 
\begin{align*}
    EML_\mathbb{R} = EL \cap \mathbb{R} \subseteq \text{computable reals} \subsetneq \mathbb{R}
\end{align*}

The paper is structured as follows: \S\ref{sec:prelim} defines EML and $\Omega_U$, and \S\ref{sec:expressibility} proves $EML_\mathbb{R} \subseteq \mathtt{Comp}_\mathbb{R}$ and derives the main inexpressibility result.

\section{Preliminaries}\label{sec:prelim}

\subsection{EML Numbers}

Let us first define Exp-Minus-Log (EML) numbers, and state some basic facts about them, all taken from \cite{Odrzywolek2026}.

\begin{definition}[Syntax of EML]
    Let $\Sigma = \{ \text{`}1\text{'}, \text{`}E\text{'}, \text{`}(\text{'}, \text{`})\text{'}, \text{`},\text{'} \}$. The set $\mathcal{E}$ of closed EML expressions be the smallest set of strings over $\Sigma$ such that `$\ 1$' $\in \mathcal{E}$, and if $\alpha, \beta \in \mathcal{E}$ then `$E(\alpha, \beta)$' $\in \mathcal{E}$
\end{definition}

\begin{definition}[Semantics]\label{def:sem}
    Define $\interpret{\cdot} : \mathcal{E} \rightarrow \mathbb{C}$ recursively:
\begin{align*}
\interpret{1} &= 1 \\
\interpret{E(\alpha, \beta)} &= \exp(\interpret{\alpha}) - \log(\interpret{\beta})
\end{align*}

Here $\exp$ is the exponential function, and $\log$ is the principal-branch complex natural logarithm defined on $\mathbb{C} \setminus \{0\}$ with imaginary part in $(-\pi, \pi]$. Given this, $\interpret{\varphi}$ is defined wherever $\interpret{\beta} \neq 0$ This ensures that $\interpret{\varphi} \in \mathbb{C}$. Let $$ EML = \{ \interpret{\varphi}: \varphi \in \mathcal{E}, \interpret{\varphi} \text{ defined} \} \subseteq \mathbb{C}$$ and let the EML-constructable reals be $EML_\mathbb{R} = EML\ \cap \mathbb{R} $
\end{definition}

As shown in \cite{Odrzywolek2026}, the following elementary constants and functions are expressible in $EML$:
\begin{enumerate}
    \item The integers and rationals,
    \item The constants $e$, $\pi$, $i$, $-1$, with $e = E(1,1)$, etc.
    \item Arithmetic operations $+$, $-$, $\times$, $/$, $x^y$, $\sqrt{x}$, along with the standard transcendental functions $\exp$, $\ln$, $\sin$, $\cos$, $\tan$, as well as the inverse trigonometric and hyperbolic functions. 
\end{enumerate}

The set of EML-Expressible numbers equals the set of complex numbers obtainable by any finite combination of $1$, $\exp$, $\log$ (as defined above), and subtraction. Equivalently, as $+$, $-$, $\times$, and $/$ are derivable, the EML-expressible numbers are the smallest subset of $\mathbb{C}$ containing $1$ closed under those operations with $\exp$ and $\log$. We now show latter is exactly Chow's EL-numbers in \cite{Chow1999}.

\subsection{Chow's EL Numbers}

Chow defines the \emph{EL numbers} in \cite{Chow1999} as follows:

\begin{definition}\label{def:elnums}
    A subfield $F$ of $\mathbb{C}$ is closed under $\exp$ and $\log$ if;  $\exp(x) \in F$ for all $x \in F$, and $\log(x) \in F$ for all nonzero $x \in F$ where $\log()$ is the branch of the natural logarithm function such that for all $x$, $- \pi < \Im(\log x) \leq \pi $. The field $\mathbb{E}$ of \emph{EL numbers} is the intersection of all subfields of $\mathbb{C}$ that are closed under $\exp$ and $\log$. 
\end{definition}

\begin{proposition}
    $$ EML = \mathbb{E} $$
\end{proposition}

\begin{proof}
    $ EML \subseteq \mathbb{E}$ is immediate from $E$ in Definition \ref{def:sem} - any number expressible in $EML$ is formed of $\exp$, $\log$, and subtraction, and so naturally also in $\mathbb{E}$.

    $\mathbb{E} \subseteq EML$ follows as every number expressible in $\exp$ and $\log$ in Definition \ref{def:elnums}, by \cite{Odrzywolek2026} may be rewritten with the following substitutions:
    \begin{align*}
        \exp(x) &= E(x,1), \\  
        \log(x) &= E(1,E(E(1,x),1))
    \end{align*} along with the following substitutions for field operations, with $1$ already given:
    \begin{align*}
        0 &= \log(1) = E(1, E(E(1,1), 1))\\
        x \times y &= E(E(1,E(E(E(1,E(E(1,E(1,x)),1)),y),1)),1) \\
        x + y &= E(1,E(E(E(1,E(E(1,E(1,E(x,1))),1)),E(y,1)),1)) \\
        -x &= E(E(1,E(E(1,E(1,E(x,1))),1)),E(E(1,1),1)) \\
        x^{-1} &= E(E(E(1,E(E(1,E(1,x)),1)),E(E(1,1),1)),1) \\
    \end{align*} 
    As shown in \cite{Odrzywolek2026}, every elementary function may also be derived in $EML$, and the conditions on $\log$ are the same. As such, we can build the full subfield of $\mathbb{C}$ with the substitutions given above from the field operations 0, 1, $+$, $\times$, and the inverses $-x$ and $x^{-1}$. So each $z \in \mathbb{E}$ is expressible as some $z' \in EML$.
\end{proof}

\subsection{Chaitin's $\Omega_U$}

Let Chaitin's $\Omega$ constants, first defined in \cite{Chaitin1975}, be defined as follows:

\begin{definition}\label{def:omega}
    For any universal prefix-free Turing machine $U$, the halting probability
\begin{align}
    \Omega_U = \sum_{p\in \text{dom}(U)} 2^{-|p|} \in (0,1)
\end{align}
\end{definition}

\begin{theorem}[Chaitin, 1975]\label{thm:omega-non-comp}
    Let $U$ be any universal prefix-free Turing machine, and $\Omega_U$ defined per Def. \ref{def:omega}. $\Omega_U$ is a real number that is left-computably enumerable, but non-computable: no Turing machine, on input $k$, will output a rational within $2^{-k}$ of $\Omega_U$ for all $k$.
\end{theorem}

\begin{proof}
    See \cite{Chaitin1975} and \cite[Thm. 3.6.9]{Downey2010-wc}.
\end{proof}

\begin{remark}[Key Step in proof of Thm. \ref{thm:omega-non-comp}]
    An oracle for the first $n$-many bits of $\Omega_U$ allows one to solve the halting problem for all programs of length $\leq n$ by dovetailing, contradicting computability.
\end{remark}

\section{$EML$ Expressibility}\label{sec:expressibility}

Now we build our reasoning towards inexpressibility within EML. To give a brief overview; first, we note the countability of $\mathcal{E}$, and then utilise facts from computable analysis \cite{Weihrauch2000} and effective lower bounds (Lemma \ref{lemma:log-bounds}) to show in Theorem \ref{thm:comp-eml-values} that all EML expressible values are computable, and then in Theorem \ref{thm:main} that Chaitin's $\Omega_U \notin \mathcal{E}$. This section is implicit in \cite{Chow1999} and is itself derivative of \cite[\S9]{TZ2009}.

\begin{lemma}[Countability of Syntax]\label{lemma:comp-syntax}
    $\mathcal{E}$ is countably infinite.
\end{lemma}

\begin{proof}
    Each $\varphi \in \mathcal{E}$ has finite length over a finite alphabet. Thus, the set of strings of arbitrary length is countable. The presence of `$1$' indicates that $\mathcal{E} \neq \emptyset$, and iterating `$E$' in line with Def. \ref{def:sem} produces infinitely many distinct expressions. 
\end{proof}

\begin{corollary}
    The set of non-EML-expressible reals has full Lebesgue measure on $\mathbb{R}$, given that it is $\mathbb{R}$ minus a countable set. 
\end{corollary}

\begin{lemma}[Computability of $\exp$ and $\log$ on computable complex numbers]\label{lemma:comp-exp-log}
    If $z$ is a computable complex number, then $\exp(z)$ is a computable complex number. 

    Similarly, if $z$ is a computable complex number with $z\neq 0$ and given a positive computable lower bound $\delta > 0$ with $|z|\geq \delta$, then $\log(z)$ is a computable complex number.
\end{lemma}

\begin{proof} 
These are standard results in computable analysis. By \cite[Thm.~6.5.2]{Weihrauch2000}, a complex analytic function with a computable power series at a computable centre is computable on any effectively given closed subset of its disc of convergence. Applied to $\exp(z) = \sum_{n \geq 0} z^n / n!$ at the origin, with infinite radius of convergence, \cite[Thm.~6.5.2]{Weihrauch2000} gives that $\exp$ is computable at every computable complex input.

The computability of $\log$ for computable input $z$ is established in Tent–Ziegler \cite[\S3–4]{TZ2009}, where the principal branch logarithm is shown to be lower-elementary computable on the slit plane $\mathbb{C} \setminus (-\infty, 0]$ via an integral representation, and lower-elementary computable functions are computable. Subtraction of computable complex numbers is computable by \cite[Thm. 4.3.9]{Weihrauch2000}.
\end{proof}

\begin{lemma}[Bounds on $\log$-arguments]
\label{lemma:log-bounds}
    Let $\varphi \in \mathcal{E}$ be a closed EML expression with $[\![\varphi]\!]$ defined. Then for every sub-expression of $\varphi$ of the form $E(\alpha, \beta)$, either $[\![\beta]\!] \in (-\infty, 0)$, or there exists a rational $\varepsilon_\beta > 0$ with $$\mathrm{dist}([\![\beta]\!], (-\infty, 0]) \geq \varepsilon_\beta$$
\end{lemma}
\begin{proof}
    The expression $\varphi$ is finite, so it has finitely many sub-expressions of the form $E(\alpha, \beta)$. Fix one. Since $\interpret{\varphi}$ is defined, $\interpret{\beta} \neq 0$ (Def. \ref{def:sem}). If $\interpret{\beta} \in (-\infty, 0)$ we are done. Otherwise $\interpret{\beta} \in \mathbb{C} \setminus (-\infty, 0]$, which is open, so $\mathrm{dist}(\interpret{\beta}, (-\infty, 0]) > 0$, and density of $\mathbb{Q}$ supplies a rational $\varepsilon_\beta \in (0, \mathrm{dist}(\interpret{\beta}, (-\infty, 0]))$.
\end{proof}

Let $\mathtt{Comp}_\mathbb{R}$ denote the computable reals, $\mathtt{Comp}_\mathbb{C}$ the computable complex numbers, and recall $EML_\mathbb{R} = EML\ \cap \mathbb{R}$. 

\begin{theorem}[Computability of EML values]
\label{thm:comp-eml-values}
    For every $\varphi \in \mathcal{E}$ with $\interpret{\varphi}$ defined, $\interpret{\varphi}$ is a computable complex number. In particular, $\mathrm{EML}_\mathbb{R} \subseteq \mathrm{Comp}_\mathbb{R}$.
\end{theorem}
\begin{proof}
    By induction on $\varphi$.

    \textbf{Base:} $[\![1]\!] = 1 \in \mathbb{Q} \subset \mathrm{Comp}_\mathbb{C}$.

    \textbf{Step:} Let $\varphi = E(\alpha, \beta)$ with $\interpret{\varphi}$ defined. Then $\interpret{\alpha}$ and $\interpret{\beta}$ are defined and, by induction, computable. By Lemma~\ref{lemma:comp-exp-log}, $\exp(\interpret{\alpha})$ is computable. By Lemma~\ref{lemma:log-bounds}, either $\interpret{\beta} \in (-\infty, 0)$ or there is a rational $\varepsilon_\beta > 0$ such that $\mathrm{dist}(\interpret{\beta}, (-\infty, 0]) \geq \varepsilon_\beta$. We then have two cases:
    \begin{enumerate}
        \item If $\interpret{\beta} \in (-\infty, 0)$, (\emph{i.e.} on the principal branch cut) then $$\log(\interpret{\beta}) = \log|\interpret{\beta}| + i\pi$$ Since $|\interpret{\beta}|$ is a positive computable real, $\log|\interpret{\beta}|$ is computable by \cite[Thm.~6.5.2]{Weihrauch2000} applied to $\log$ on $\mathbb{R}_{>0}$, and adding $i\pi$ preserves computability.
        \item Otherwise, by Lemma~\ref{lemma:comp-exp-log} we have that $\log(\interpret{\beta})$ is computable.
    \end{enumerate}
    Subtraction is computable, so $\interpret{\varphi} = \exp(\interpret{\alpha}) - \log(\interpret{\beta})$ is computable.
\end{proof}

It should be noted that Theorem \ref{thm:comp-eml-values} in no way claims to be \emph{uniformly} computable - in fact, the central disjunction is not computable (not even r.e.), \emph{i.e.} there is not a singular Turing machine that will perform this computation. This is fine for the pointwise arguments here, but needs to be noted as a limitation.

\begin{theorem}\label{thm:main}
$$\Omega_U \notin EML_\mathbb{R}$$
\end{theorem}
\begin{proof}
    By Theorem \ref{thm:comp-eml-values}, $EML_\mathbb{R} \subseteq \mathtt{Comp}_\mathbb{R}$, and by Theorem \ref{thm:omega-non-comp}, $\Omega_U \in \mathbb{R} \setminus \mathtt{Comp}_\mathbb{R}$. Hence $\Omega_U \notin EML_\mathbb{R}$. Similarly, for any non-computable real $r$, $r\notin \mathtt{Comp}_\mathbb{R} \supseteq EML_\mathbb{R}$, therefore $\nexists \varphi \in \mathcal{E}$ such that $\interpret{\varphi} = r$.
\end{proof}
    
\begin{corollary}
    The set of non-computable reals, which is of full Lebesgue measure, is entirely disjoint from EML. In addition to the above, this includes natural sub-classes such as:

    \begin{itemize}
        \item The binary expansion of the halting set $\{ 0.\chi_{HALT}(1)\chi_{HALT}(2)\chi_{HALT}(3)\ldots \}$
        \item The limits of Specker sequences (see \cite{Weihrauch2000})
        \item Any random real in the sense of Martin-L\"of (see \cite{Downey2010-wc}).
    \end{itemize}
\end{corollary}

\bibliographystyle{acm}
\bibliography{bib}

\end{document}